\documentclass[12pt]{amsart}
\usepackage[english]{babel}
\usepackage[utf8]{inputenc}
\usepackage{amsmath, amssymb, amsthm, amsfonts, amsxtra, stmaryrd, enumerate, graphicx}
\usepackage[colorinlistoftodos]{todonotes}
\usepackage{placeins}

\newtheorem{theorem}{Theorem}[section]
\newtheorem{corollary}{Corollary}[theorem]
\newtheorem{lemma}[theorem]{Lemma}
\newtheorem{problem}[theorem]{Problem}
\newtheorem{claim}[theorem]{Claim}
\newtheorem{proposition}[theorem]{Proposition}

\newtheorem{remark}[theorem]{Remark}
\newtheorem{fact}[theorem]{Fact}
\usepackage{tikz-cd}
\newtheorem{definition}[theorem]{Definition}

\theoremstyle{definition}

\def\Ind#1#2{#1\setbox0=\hbox{$#1x$}\kern\wd0\hbox to 0pt{\hss$#1\mid$\hss}
\lower.9\ht0\hbox to 0pt{\hss$#1\smile$\hss}\kern\wd0}

\def\ind{\mathop{\mathpalette\Ind{}}}

\def\notind#1#2{#1\setbox0=\hbox{$#1x$}\kern\wd0
\hbox to 0pt{\mathchardef\nn=12854\hss$#1\nn$\kern1.4\wd0\hss}
\hbox to 0pt{\hss$#1\mid$\hss}\lower.9\ht0 \hbox to 0pt{\hss$#1\smile$\hss}\kern\wd0}

\usepackage{etoolbox}
\patchcmd{\subsection}{-.5em}{.5em}{}{}

\title{On $\mathrm{NSOP}_{2}$ Theories}
\author{Scott Mutchnik}

\begin{document}

\begin{abstract}
Answering a question of Džamonja and Shelah, we show that every $\mathrm{NSOP}_{2}$ theory is $\mathrm{NSOP}_{1}$.

\end{abstract}
\maketitle

\tableofcontents

\newpage

\section{Introduction}
One of the most exciting areas of research in modern model theory is the classification along various dividing lines of non-simple but otherwise tame theories, especially $\mathrm{NSOP}_{n}$ theories for $1 \leq n \leq 3$. The first two of these properties, introduced in
\cite{DS04}, require the nonexistence of certain trees:

\begin{definition}
A theory $T$ is $\mathrm{NSOP}_{1}$ if there does not exist a formula $\varphi(x, y)$ and tuples $\{b_{\eta}\}_{\eta \in 2^{<\omega}}$ so that $\{\varphi(x, b_{\sigma \upharpoonleft n})\}_{n \in \omega}$ is consistent for any $\sigma \in 2^{\omega}$, but for any $\eta_{2} \unrhd \eta_{1} \smallfrown \langle 0\rangle$, $\{\varphi(x, b_{\eta_{2}}), \varphi(x, b_{\eta_{1} \smallfrown \langle 1\rangle})\}$ is inconsistent. Otherwise it is $\mathrm{SOP}_{1}$.
\end{definition}

\begin{definition}
A theory $T$ is $\mathrm{NSOP}_{2}$  if there does not exist a formula $\varphi(x, y)$ and tuples $\{b_{\eta}\}_{\eta \in 2^{<\omega}}$ so that $\{\varphi(x, b_{\sigma \upharpoonleft n})\}_{n \in \omega}$ is consistent for any $\sigma \in 2^{\omega}$, but for incomparable $\eta_{1}$ and $\eta_{2}$, $\{\varphi(x, b_{\eta_{1}}), \varphi(x, b_{\eta_{2}})\}$ is inconsistent. Otherwise it is $\mathrm{SOP}_{2}$.
\end{definition}

The property $\mathrm{NSOP}_{3}$ is introduced in \cite{She95} as part of a family of notions $\mathrm{NSOP}_{n}$ for $n \geq 3$:

\begin{definition}
A theory $T$ is $\mathrm{NSOP}_{n}$ (that is, does not have the \emph{n-strong order property}) if there is no definable relation $R(x_{1}, x_{2})$ with no $n$-cycles, but with tuples $\{a_{i}\}_{i \in \omega}$ with $\models R(a_{i}, a_{j})$ for $i <j$. Otherwise it is $\mathrm{SOP}_{n}$.
\end{definition}

\begin{fact}
\emph{(\cite{She95}, \cite{DS04})} Simple theories are $\mathrm{NSOP}_{1}$, and $\mathrm{NSOP}_{n}$ theories are $\mathrm{NSOP}_{m}$ for $n \leq m$.
\end{fact}

In \cite{SD00} it is shown that $T^{*}_{feq}$, the model companion of the theory of parametrized equivalence relations, is $\mathrm{NSOP}_{1}$ but not simple; a limited number of further examples have since been found by various authors. Yet the main problem, posed by Džamonja and  Shelah in \cite{DS04}, has remained unsolved:

\begin{problem}
Are all $\mathrm{NSOP}_{3}$ theories $\mathrm{NSOP}_{2}$? Are all $\mathrm{NSOP}_{2}$ theories $\mathrm{NSOP}_{1}$?
\end{problem}

In this paper we answer the latter question in the positive:

\begin{theorem}\label{3-main}
All $\mathrm{NSOP}_{2}$ theories are $\mathrm{NSOP}_{1}$.
\end{theorem}

One reason for the significance of this problem comes from Shelah and Usvyatsov's proposal in \cite{SD00} to characterize classes of theories both internally in terms of the structure of their sufficiently saturated models, and externally in terms of orders on theories. The $\mathrm{NSOP}_{2}$ theories have a deep external characterization: under the generalized continuum hypothesis, Džamonja and Shelah \cite{DS04} show that maximality in the order $\lhd^{*}$, an order related to the Keisler order, implies a combinatorial property related to $\mathrm{SOP}_{2}$, which Shelah and Usvyatsov then show in \cite{SD00} to be the same as $\mathrm{SOP}_{2}$; later, Malliaris and Shelah in \cite{MS17} show the equivalence between $\mathrm{SOP}_{2}$ and $\lhd^{*}$-maximality under the generalized continuum hypothesis. On the other hand, $\mathrm{NSOP}_{1}$ theories can be characterized internally not only in terms of trees, but through the theory of independence, in analogy with stability theory. It is well known that simple theories are characterized as those theories where forking and dividing behave in certain ways as they do in stable theories; for example, symmetry of forking characterizes simple theories. In \cite{KR17}, Kaplan and Ramsey show that \textit{Kim-forking}, or forking witnessed by invariant Morley sequences, is the correct way of extending the theory of forking to $\mathrm{NSOP}_{1}$ theories from simple theories. By relaxing the requirement of base monotonicity, they extend the Kim-Pillay characterization of simple theories in terms of the existence of abstract independence relations to $\mathrm{NSOP}_{1}$ theories, and, more concretely, characterize $\mathrm{NSOP}_{1}$ theories by the symmetry of Kim-independence, by the independence theorem for Kim-independence, and by a variant of Kim's lemma in simple theories, asserting that \textit{Kim-dividing} of a formula, rather than dividing, is witnessed by any invariant Morley sequence. Our result that $\mathrm{NSOP}_{1}$ theories coincide with $\mathrm{NSOP}_{2}$ theories therefore shows a surprising agreement between dividing lines related to Keisler's order and dividing lines related to independence.

We outline the paper and give a word on the strategy for the proof. In section 3, we develop in general theories a version of a construction used by Chernikov and Kaplan in \cite{CK09} to study forking and dividing in $\mathrm{NTP}_{2}$ theories. In \cite{A09}, Adler initiated the study of abstract relations between sets in a model, generalizing some of the properties of forking-independence, coheirs, and other concrete relations from model theory, and provided a set of potential axioms for these relations\footnote{Other than Adler's work in \cite{A09} and Conant's work on free amalgamation theories in \cite{Co15}, an additional observation which ultimately led us to the proof of this result is found in \cite{D19}, where d'Elb\'ee proposes the problem of explaining the apparent ubiquity of additional independence relations with no known concrete model-theoretic independence relations in $\mathrm{NSOP}_{1}$ theories, such as \textit{strong independence} existing alongside Kim-independence in the theory $\mathrm{ACFG}$ (introduced as part of a more general class in \cite{D18}) of algebraically closed fields with a generic additive subgroup. He observes that just as in the case of free amalgamation of generic functional structures in \cite{KR18} or generic incidence structures in \cite{CoK19}, these stronger independence relations can be used to prove the equivalence of forking and dividing for complete types in many known $\mathrm{NSOP}_{1}$ theories. Before proving Theorem 1.6, we gave some very weak axioms (including \textit{stationarity}, a feature of the examples considered by \cite{D19}) for abstract relations between sets over a model, which appeared to be very common in $\mathrm{NSOP}$ theories including strictly $\mathrm{NSOP}_{1}$ theories and $\mathrm{NSOP}_{4}$ theories, and proved that theories with such a relation could not be $\mathrm{NSOP}_{2}$; instead of considering Morley sequences in canonical coheirs as in the below, we used $\ind$-independent sequences for the abstract relation $\ind$, in the sense of Definition 7.5 of \cite{Co15}. Note also that the property \textit{quasi-strong finite character} considered below is a property of the examples in \cite{D19}.}. We notice that the construction of Chernikov and Kaplan can be relativized to relations between sets satisfying certain axioms, obtaining new relations between sets from old ones, and iterate this construction to obtain a canonical class of coheirs in any theory.

In section 4, we study this canonical class of coheirs in $\mathrm{NSOP}_{2}$ theories. Before the development of Kaplan and Ramsey's theory of Kim-independence in $\mathrm{NSOP}_{1}$ theories in \cite{KR17}, Chernikov \cite{C15} proposed finding a theory of independence for $\mathrm{NSOP}_{2}$ theories, and the proof of our main result comes from our efforts to answer this proposal. Just as in \cite{CK09}, Chernikov and Kaplan's construction gives maximal classes in the \textit{dividing order} of Ben Yaacov and Chernikov \cite{BYC07}, we show that in $\mathrm{NSOP}_{2}$ theories our variant of this construction gives \textit{minimal} classes in the restriction of this order to coheir Morley sequences, proving an analogue of Kim's lemma. As a by-product of this construction, we also initiate the theory of independence in a class related to the $\mathrm{NATP}$ theories of Ahn and Kim \cite{AK20}, the study of which was further developed by Ahn, Kim and Lee in \cite{AK21}, showing that under this assumption Kim-forking and Kim-dividing coincide for coheir Morley sequences. (See \cite{Kr} for the question of finding an analogue for $\mathrm{NSOP}_{1}$ theories of the role that $\mathrm{NTP}_{2}$ theories play relative to simple theories, and developing Kim-independence in that analogue; that Kim-forking coincides with Kim-dividing for coheir Morley sequences in a related class gives us preliminary evidence that $\mathrm{NATP}$ completes this analogy.)

In section 5, we investigate behavior similar to $\mathrm{NSOP}_{1}$ theories in $\mathrm{NSOP}_{2}$ theories. We introduce the notion of \textit{Conant-independence}, which will generalize the relation $A \ind^{a}_{M}B$ defined by $\mathrm{acl}(MA) \cap \mathrm{acl}(MB)=M$ in the \textit{free amalgamation theories} introduced by Conant \cite{Co15} (based on concepts used to study the isometry groups of Urysohn spheres in \cite{ZT12}); see the following section. While it will end up coinciding with Kim-independence in our case, we studied a version of Conant-independence in a potentially strictly $\mathrm{NSOP}_{1}$, potentially $\mathrm{SOP}_{3}$ generalization of free amalgamation theories in \cite{GFA}. Conant-independence in $\mathrm{NSOP}_{2}$ theories can be defined as Kim-independence relative to canonical Morley sequences, just as $\ind^{a}$ is Kim-independence relative to free amalgamation Morley sequences (as in lemma 7.7 of \cite{Co15}); it can also be defined by forcing Kim's lemma on Kim-independence, requiring a formula to divide with respect to \textit{every} Morley sequence instead of just one, as suggested in tentative remarks of Kim in \cite{K09} in his discussion of \textit{strong dividing} in \textit{subtle theories}. We show that many of Ramsey and Kaplan's arguments on Kim-independence in $\mathrm{NSOP}_{1}$ theories in \cite{KR17} can be generalized to Conant-independence in $\mathrm{NSOP}_{2}$ theories, including a chain condition, symmetry and a weak independnece theorem. (But as is apparent in \cite{Co15} and \cite{GFA}, similar behavior can occur in a $\mathrm{SOP}_{3}$ theory, which is why the following section is essential to the proof of our main result.)

In section 6, we conclude the proof of Theorem \ref{3-main}. One consequence of Conant's free amalgamation axioms (say, the freedom, closure and stationarity axioms, in Defintion 2.1 in \cite{Co15}) is the following:

Let $\ind$ denote free amalgamation and $A_{1} \ind^{a}_{M} B$, $A_{2} \ind^{a}_{M} C$, and $B \ind_{M} C$ with $A_{1} \equiv_{M} A_{2} $. Then there is some $A \ind^{a}_{M} BC$ with $A \equiv_{MB} A_{1}$ and $A \equiv_{MC} A_{2}$.

We will have shown in the prior section that Conant-independence is symmetric, and that a similar fact holds, roughly, when replacing free amalgamation with canonical coheirs and $\ind^{a}$ with Conant-independence. Conant shows in \cite{Co15} that \textit{modular} free amalgamation theories must either be simple or $\mathrm{SOP}_{3}$ (see \cite{EW09} for a related result on countably categorical Hrushovski constructions), starting with a failure of forking-independence to coincide with $\ind^{a}$ (because forking-independence cannot be symmetric unless a theory is simple) and using the above fact to build up a configuration giving $\mathrm{SOP}_{3}$. Starting, analogously, with the assumption that an $\mathrm{NSOP}_{2}$ theory $T$ is $\mathrm{SOP}_{1}$, so Kim-dividing independence is not symmetric and therefore fails to coincide with Conant-independence, we simulate Conant's construction of an instance of $\mathrm{SOP}_{3}$. In short, we show that a $\mathrm{NSOP}_{2}$ theory is either $\mathrm{NSOP}_{1}$ or $\mathrm{SOP}_{3}$. But a $\mathrm{NSOP}_{2}$ theory is of course not $\mathrm{SOP}_{3}$, so it must be $\mathrm{NSOP}_{1}$.

\section{Preliminaries}
We let $a, b, c, d, e, A, B, C$ denote sets, potentially with an enumeration depending on context, and $x, y, z, X, Y, Z$ denote tuples of variables. We let $\mathbb{M}$ denote a sufficiently saturated model of a theory $T$ and let $M$ denote an elementary submodel. We write $AB$ to denote the union (or concatenation) of the sets $A$ and $B$, and write $I$, $J$, etc. for infinite sequences (or sometimes trees) of tuples or an infinite linearly ordered set.

\textbf{Relations between sets}

Roughly following the axioms for abstract independence relations in \cite{A09}, as well as others that are standard in the literature, we define the following axioms for relations $A \ind_{M} B$ between sets over a model:

Invariance: For all $\sigma \in \mathrm{Aut}(\mathbb{M})$, $A \ind_{M} B$ implies $\sigma(A) \ind_{\sigma(M)} \sigma(B)$.

Full existence: For $M \subseteq A, B \subseteq \mathbb{M}$, there is always some $A' \equiv_{M} A$ with $A \ind_{M} B$.

Left extension: If $A\ind_{M} B $ and $A \subseteq C$, there is some $B' \equiv_{A} B$ with  $C\ind_{M} B' $.

Right extension: If $A\ind_{M} B $ and $B \subseteq C$, there is some $A' \equiv_{B} A$ with  $A'\ind_{M} C $.

Left monotonicity: If $A\ind_{M} B $ and $M \subseteq A' \subseteq A$, then $A'\ind_{M} B $

Right monotonicity: If $A\ind_{M} B $ and $M \subseteq B' \subseteq B$, then $A\ind_{M} B' $

(We will refer to the two previous properties, taken together, as monotonicity.)

Symmetry: If $A \ind_{M} B$ then $B \ind_{M} A$

\textbf{Coheirs and Morley sequences}

A global type $p$ is a complete type over $\mathbb{M}$. For $p \in S(A)$ for $M \subseteq A$, we say $p$ is \textit{finitely satisfiable} over $M$ or a \textit{coheir extension} of its restriction to $M$ if every formula in $p$ is satisfiable in $M$. Global types $p$ finitely satisfiable in $M$ are \textit{invariant} over $M$: whether $\varphi(x, b)$ belongs to $p$ for $\varphi$ a formula without parameters, depends only on the type of the parameter $b$ over $M$. We write $a\ind^{u}_{M}b$ to denote that $\mathrm{tp}(a/Mb)$ is finitely satisfiable in $M$. We let $a \ind_{M}^{h} b$ denote $b\ind^{u}_{M} a$. The relation $\ind^{u}$ (over models) is well-known to satisfy all of the above properties other than symmetry. We say $\{b_{i}\}_{i \in I}$, for $I$ potentially finite, is a \textit{coheir sequence} over $M$ if $b_{i}\ind^{u}_{M} b_{<i}$ for $i \in I$. We say a coheir sequence $\{b_{i}\}_{i \in I}$, for $I$ infinite, is moreover a \textit{coheir Morley sequence} over $M$ if there is a fixed global type $p(x)$ finitely satisfiable in $M$ so that $b_{i} \models p(x)|_{\{Mb_{j}\}_{j < i}}$ for $i \in I$.  The type of a coheir Morley sequence over $M$ (indexed by a given set) is well-known to depend only on $p(x)$, and coheir Morley sequences are known to be indiscernible; the type of a coheir sequence over $M$ depends only on the global coheirs over $M$ extending the $\mathrm{tp}(b_{i}/Mb_{<i})$.

\textbf{$\mathrm{NSOP}_{1}$ theories and Kim-dividing}

In this paper we use nonstandard terminology: Kim-dividing, etc. are defined in terms of Morley sequences in \textit{invariant} types over $M$ rather than \textit{finitely satisfiable} types over $M$ in \cite{KR17}. The reason why we do this is that $\ind^{u}$ is known to satisfy left extension. This will do us no harm for our main result, though when we briefly consider Kim-forking in some $\mathrm{NATP}$ theories, we will note the nonstandard usage.

\begin{definition}
A formula $\varphi(x, b)$ \emph{Kim-divides} over $M$ if there is an coheir Morley sequence $\{b_{i}\}_{i \in \omega}$ starting with $b$ so that $\{\varphi(x, b_{i})\}_{i \in \omega}$ is inconsistent (equivalently, $k$-inconsistent for some $k$: any subset of size $k$ is inconsistent). A formula  $\varphi(x, b)$ \emph{Kim-forks} over $M$ if it implies a (finite) disjunction of formulas Kim-dividing over $M$. We write $a \ind^{Kd}_{M} b$, and say that $a$ is \emph{Kim-dividing independent} from $b$ over $M$ if $\mathrm{tp}(a/Mb)$ does not include any formulas Kim-dividing over $M$.
\end{definition}

The following follows directly from Proposition 5.2 of \cite{CR15}; see also Proposition 3.22 of \cite{KR17} (where the evident argument for the version for invariant types is given) and Theorem 5.16 of \cite{KR17} for the full symmetry characterization of $\mathrm{NSOP}_{1}$.

\begin{fact}\label{3-kimsymm}
Symmetry of $\ind^{Kd}$ implies $\mathrm{NSOP}_{1}$.
\end{fact}

\textbf{$\mathrm{NSOP}_{2}$ theories}

A characterization of $\mathrm{SOP}_{2}$ as \textit{$k$-$\mathrm{TP}_{1}$} was proven by Kim and Kim in \cite{KK11}, where they also introduce the notion of \textit{weak $k$-$\mathrm{TP}_{1}$}, prove that it implies $\mathrm{SOP}_{1}$, and conjecture that it also implies $\mathrm{SOP}_{2}$:

\begin{definition}
The theory $T$ has \emph{weak $k$-$\mathrm{TP}_{1}$} if there exists a formula $\varphi(x, y)$ and tuples $\{b_{\eta}\}_{\eta \in \omega^{<\omega}}$ so that $\{\varphi(x, b_{\sigma \upharpoonleft n})\}_{n \in \omega}$ is consistent for any $\sigma \in \omega^{\omega}$, but for pairwise incomparable $\eta_{1} \ldots, \eta_{k} \in \omega^{< \omega}$ with common meet, $\{\varphi(x, b_{\eta_{i}})\}_{i=1}^{k}$ is inconsistent.
\end{definition}

Later, Chernikov and Ramsey, in Theorem 4.8 of \cite{CR15}, claim to show that weak $k$-$\mathrm{TP}_{1}$ implies $\mathrm{SOP}_{2}$, but their proof is incorrect; the embedded tree $\{b_{\eta}\}_{\eta \in \omega^{<\omega}}$ in the proof of that theorem is not actually strongly indiscernible over the parameter set $C$. In an earlier version of this paper, we used this result. In this section, we will introduce an equivalent form of $\mathrm{SOP}_{2}$ that will suffice for our argument, and use the same method as \cite{CR15} to give a proof that will work to show this equivalence despite failing for weak $k$-$\mathrm{TP}_{1}$.

\begin{definition}\label{3-descendingcomb}(Proposition 2.51, item IIIa, [102]). A list $\eta_{1}, \ldots, \eta_{n} \in \omega^{<\omega}$ is a \emph{descending comb} if and only if it is an antichain so that $\eta_{1} <_{\mathrm{lex}} \ldots <_{\mathrm{lex}} \eta_{n}$, and so that, for $1 \leq k < n$,  $\eta_{1} \wedge \ldots \wedge \eta_{k+1} \lhd \eta_{1} \wedge \ldots \wedge \eta_{k}.$
    
\end{definition}

So for example, all descending combs of length $n$ have the same quantifier-free type in the language $\{<_{\mathrm{lex}}, \lhd, \wedge \}$ as the descending comb $\langle 0\rangle^{n-1} \smallfrown \langle 1\rangle, \ldots \langle 1\rangle$; meanwhile, $\langle 00 \rangle, \langle 01 \rangle, \langle 10 \rangle, \langle 11 \rangle$ is an example of a lexicographically ordered antichain that is not a descending comb.

\begin{definition}
   (Definitions 11 and 12, \cite{TT12}) For tuples $\overline{\eta}, \overline{\eta}' \in \omega^{<\omega}$ of elements of $\omega^{<\omega}$, we write $\overline{\eta} \sim_{0} \overline{\eta}'$ to mean that $\overline{\eta}$ has the same quantifier-free type in the language $\{<_{\mathrm{lex}}, \lhd, \wedge \}$ as $\overline{\eta}'$. For $(b_{\eta})_{\eta \in \omega^{<\omega}}$ a tree-indexed set of tuples and $\overline{\eta} = \eta_{1}, \ldots, \eta_{n} \in \omega^{<\omega}$ an $n$-tuple of elements of $\omega^{<\omega}$, we write $b_{\overline{\eta}}=: b_{\eta_{1}}\ldots b_{\eta_{n}}$, and call $(b_{\eta})_{\eta \in \omega^{<\omega}}$ \emph{strongly indiscernible} over a set $A$ if for all tuples  $\overline{\eta}, \overline{\eta}' \in \omega^{<\omega}$ of elements of $\omega^{<\omega}$ with $\overline{\eta} \sim_{0} \overline{\eta}'$, $b_{\overline{\eta}} \equiv_{A} b_{\overline{\eta}'}$.
\end{definition}

\begin{fact} \label{3-tt} (Theorem 16, \cite{TT12}; see \cite{Sc15} for an alternate proof) Let $(b_{\eta})_{\eta \in \omega^{<\omega}}$ be a tree-indexed set of tuples, and $A$ a set. Then there is $(c_{\eta})_{\eta \in \omega^{<\omega}}$ strongly indiscernible over $A$ so that for any tuple $\overline{\eta} \in \omega^{< \omega}$ of elements of $\omega^{<\omega}$ and $\varphi(x) \in L(A)$, if $\models \varphi(b_{\overline{\eta'}})$ for all $\overline{\eta}' \sim_{0} \eta$, then $\models \varphi(c_{\overline{\eta}})$.
\end{fact}

\begin{definition}
The theory $T$ has $k$-$\mathrm{DCTP}_{1}$ if there exists a formula $\varphi(x, y)$ and tuples $\{b_{\eta}\}_{\eta \in \omega^{<\omega}}$ so that $\{\varphi(x, b_{\sigma \upharpoonleft n})\}_{n \in \omega}$ is consistent for any $\sigma \in \omega^{\omega}$, but for any descending comb  $\eta_{1} \ldots, \eta_{k} \in \omega^{< \omega}$, $\{\varphi(x, b_{\eta_{i}})\}_{i=1}^{k}$ is inconsistent.

\end{definition}

\begin{lemma}\label{3-dctp1}
For any $k> 1$, a theory has $\mathrm{SOP}_{2}$ if and only if it has $k$-$\mathrm{DCTP}_{1}$.
\end{lemma}

\begin{proof}
($\Rightarrow$) The property $2$-$\mathrm{DCTP}_{1}$ follows directly from Fact 4.2, \cite{CR15}.

($\Leftarrow$) We follow the proof of theorem 4.8 of (\cite{CR15}), which is incorrect for the claimed result. Let $\{b_{\eta}\}_{\eta \in \omega^{<\omega}}$ witness $\mathrm{DCTP}_{1}$ with the formula $\varphi(x, y)$. By fact \ref{3-tt}, we can assume $\{b_{\eta}\}_{\eta \in \omega^{<\omega}}$ is strongly indiscernible (as paths and descending combs are preserved under $\sim_{0}$-equivalence), and will produce a witness to $\mathrm{SOP}_{2}$. Let $\eta_{i} = \langle 0 \rangle^{i} \smallfrown \langle 1 \rangle $ (so that, say, $\eta_{n}, \ldots, \eta_{0}$ will form a descending comb), and let $n$ be maximal so that 

$$\{\varphi(x, b_{\eta_{i} \smallfrown \langle 0 \rangle^{\alpha}}): i < n, \alpha < \omega \}$$ is consistent; by consistency of the paths, $n$ will be at least $1$, and by inconsistency of descending combs of size $k$, $n$ will be at most $k$. Let $C = \{ b_{\eta_{i} \smallfrown \langle 0 \rangle^{\alpha}}: i < n-1, \alpha < \omega \}$. We see that, say, $\mu = \langle 0 \rangle_{n-1}$ sits strictly above the meets of any two or more of the $\eta_{i}$ for $i < n-1$ in the order $\lhd$, and is incomparable to and lexicographically to the left of $\eta^{n-2}$ when $n > 1$, so the appropriately tree-indexed subset $\{c_{\eta}\}_{\eta \in \omega^{<\omega}}$ of $\{b_{\eta}\}_{\eta \in \omega^{<\omega}}$ consisting of those $b_{\eta}$ with $\mu \unlhd \eta$ (that is, where $c_{\eta}=b_{\mu \smallfrown \eta}$) really is strongly indiscernible over $C$. By strong indiscernibility of $\{b_{\eta}\}_{\eta \in \omega^{< \omega}}$ and the fact that $\{\varphi(x, b_{\eta_{i} \smallfrown \langle 0 \rangle^{\alpha}}): i < n, \alpha < \omega \}$ is consistent, $ \{\varphi(x, c_{\langle 0 \rangle \smallfrown \langle 0 \rangle^{\alpha}}): i < n, \alpha < \omega \}\cup \{\varphi(x, c): c \in C\}$ is consistent; let $d$ realize it, and by Ramsey, compactness and an automorphism over $C$, we can assume $\{c_{\langle 0 \rangle \smallfrown \langle 0 \rangle^{\alpha}}\}_{\alpha < \omega}$ is indiscernible over $dC$. On the other hand, for $p(y, \overline{z})= \mathrm{tp}(d, \{c_{\langle 0 \rangle \smallfrown \langle 0 \rangle^{\alpha}}\}_{\alpha < \omega}/ C)$, we see that $p(y, \{c_{\langle 0 \rangle \smallfrown \langle 0 \rangle^{\alpha}}\}_{\alpha < \omega}) \cup p(y, \{c_{\langle 1 \rangle \smallfrown \langle 0 \rangle^{\alpha}}\}_{\alpha < \omega})$ is inconsistent, by strong indiscernibility of $\{b_{\eta}\}_{\eta \in \omega^{< \omega}}$ and inconsistency (by maximality of $n$) of $\{\varphi(x, b_{\eta_{i} \smallfrown \langle 0 \rangle^{\alpha}}): i \leq n, \alpha < \omega \}$ (noting that, say, $p(y, \{c_{\langle 0 \rangle \smallfrown \langle 0 \rangle^{\alpha}}\}_{\alpha < \omega})$ contains $\{ b_{\eta_{i} \smallfrown \langle 0 \rangle^{\alpha}}: i < n-1, \alpha < \omega \}$). This is exactly what the ``path collapse lemma," Lemma 4.6 of \cite{CR15}, tells us that we need to obtain $\mathrm{SOP}_{2}$.
\end{proof}

Though the proof of Theorem 4.8 of \cite{CR15} is incorrect, that theorem (albeit, not a ``local" version) will be a corollary of our main result, Theorem \ref{3-main}, and the result of \cite{KK11} that weak $k$-$\mathrm{TP}_{1}$ implies $\mathrm{SOP}_{1}$. (Note that $\mathrm{SOP}_{2}$ is just weak $2$-$\mathrm{TP}_{1}$).

\begin{corollary}
    (to Theorem \ref{3-main}) For any $k$, a theory has weak $k$-$\mathrm{TP}_{1}$ if and only if it has $\mathrm{SOP}_{2}$.
\end{corollary}

\section{Canonical coheirs in any theory}

The following section will require no assumptions on $T$. Iterating a similar construction to the one used by Chernikov and Kaplan in \cite{CK09} to prove the equivalence of forking and dividing for formulas in $\mathrm{NTP}_{2}$ theories, we will contruct a canonical class of coheir extensions in any theory. This class will end up satisfying a variant of the ``Kim's lemma for Kim-dividing" in $\mathrm{NSOP}_{1}$ theories (Theorem 3.16 of \cite{KR17}) when considered in a $\mathrm{NSOP}_{2}$ theory.

\begin{proposition}\label{3-ck}
Let $T$ be any theory. Consider relations $\ind$ between sets over a model that are stronger that $\ind^{h}$, satisfy invariance, monotonicity, full existence and right extension, and satisfy the coheir chain condition: if $a \ind_{M}b$ and $I =\{b_{i}\}_{i \in \omega}$ is a coheir Morley sequence starting with $b$, then there is some $I' \equiv_{M} I $ with $a \ind_{M} I'$ and each term of $I'$ satisfying $\mathrm{tp}(b/Ma)$. There is a weakest such relation $\ind^{\mathrm{CK}}$.
\end{proposition}

The ``weakest" clause is not necessary for the main result, but we include it anyway to show our construction is canonical.

We start by relativizing the notions of Kim-dividing, Kim-forking, and \textit{quasi-dividing} (Definition 3.2 of \cite{CK09}) to an $M$-invariant ideal on the definable subsets of $\mathbb{M}$.

\begin{definition}
Let $\mathcal{I}$ be an $M$-invariant ideal on the definable subsets of $\mathbb{M}$. A formula $\varphi(x, b)$ \emph{$\mathcal{I}$-Kim-divides} over $M$ if there is a coheir Morley sequence $\{b_{i}\}_{i \in \omega}$ starting with $b$ so that for some $k$, the intersection of some (any) $k$-element subset of $\{\varphi(\mathbb{M}, b_{i})\}_{i \in \omega}$ belongs to $\mathcal{I}$. We say $\varphi(x, b)$ \emph{$\mathcal{I}$-Kim-forks} over $\mathbb{M}$ if it implies a (finite) disjunction of formulas $\mathcal{I}$-Kim-dividing over $M$. We say $\varphi(x, b)$ \emph{$\mathcal{I}$-quasi-divides} over $M$ if there are $b_{1}, \ldots, b_{n}$ with $b \equiv_{M} b_{i}$ so that $\cap_{i=1}^{n} \varphi(\mathbb{M}, b_{i}) \in \mathcal{I}$.

We say $\varphi(x, b) \vdash^{\mathcal{I}} \psi(x, c)$ if $\varphi(\mathbb{M}, b) \backslash \psi(\mathbb{M}, c)  \in \mathcal{I}$.  
\end{definition}

The proof of the following lemma is adapted straightforwardly from the proof of the ``broom lemma" of Chernikov and Kaplan (Lemma 3.1 of \cite{CK09})\footnote{Alex Kruckman, in a personal communication with the author, discussed an alternative to this proof for showing the properness of the ideal corresponding to the independence result of $\ind^{\mathrm{CK}}$, with the broom lemma as a corollary, which works for invariant Morley sequences as well as coheir Morley sequences; it is based on unpublished work of James Hanson on the concept of ``fracturing," a generalization of quasi-forking and quasi-dividing.}. For the convenience of the reader we give a simplified proof of the modified version; note that this version is just a rephrasing in terms of ideals of Lemma 4.19 in \cite{Che14}:

\begin{lemma}\label{3-broom}
(``$\mathcal{I}$-broom lemma") Suppose 

$$\alpha(x, e) \vdash^{\mathcal{I}} \psi(x, c) \vee \bigvee_{i=1}^{N}\varphi_{i}(x, a_{i})$$ with $\varphi_{i}(x, a_{i})$ $\mathcal{I}$-Kim-dividing over $M$ with respect to $P(x)$ and $c \ind^{u}_{M}a_{1}\ldots a_{N}$. Then there are some $e_{1}, \ldots e_{m}$ with $e_{i} \equiv_{M} e$ so that $\bigwedge_{i=1}^{m} \alpha(x, e_{i}) \vdash^{\mathcal{I}} \psi(x, c) $. In particular, $\mathcal{I}$-Kim-forking implies $\mathcal{I}$-quasi-dividing over $M$.
\end{lemma}

\begin{proof}
We need the following claim:

\begin{claim} \label{3-lext}
Let $a^{1}, \ldots, a^{n}$ begin a coheir Morley sequence in a global type $q$ finitely satisfiable over $M$. Let $a \equiv_{M} a^{i}$ and let $b$ be any tuple. Then there are $b^{1}, \ldots, b^{n}$ so that $b^{1}a^{1}, \ldots, b^{n}a^{n}$ begin a coheir Morley sequence and $b^{i}a^{i} \equiv_{M} ba$. (The same is true for Coheir morley sequences themselves, rather than just their initial segments).
\end{claim}

\begin{proof}
Left extension for $\ind^{u}$ gives a global type $r$ finitely satisfiable over $M$ extending both $q$ and $\mathrm{tp}(ab/M)$. Now take a coheir Morley sequence in $r$ and apply an automorphism. The parenthetical is similar.
\end{proof}

Now we can prove the lemma by induction on $N$. Write $\bigvee_{i=1}^{N-1}\varphi_{i}(x, a_{i})$ as $\varphi(x, b)$, and let $a=a_{N}$. Let $p$ be a global coheir extension of $\mathrm{tp}(c/Mba)$. Let $(a^{i})_{i=1}^{n}$ be such that $a^{i} \ind^{u} a^{i-1}, \ldots, a^{1}$ and $a^{i} \equiv_{M} a$ for $1 \leq i \leq n$ and $\wedge_{i =1}^{n} \varphi_{N}(x, a^{i}) \vdash^{\mathcal{I}}\bot$. By the claim, find $b^{1}, \ldots b^{n}$ so that $a^{i}b^{i} \ind^{u} a^{i-1}b^{i-1} \ldots a^{1}b^{1}$ and $a^{i}b^{i} \equiv_{M} ab$ for $1 \leq i \leq n$. Then we can assume $c \models p|_{Maba^{1}b^{1} \ldots a^{n}b^{n}}$. From $c \ind^{u}_{M} a^{1}b^{1} \ldots a^{n}b^{n}$, together with $a^{i}b^{i} \ind^{u} a^{i-1}b^{i-1} \ldots a^{1}b^{1}$ for $1 \leq i \leq n$, it is easy to check $ca^{i+1}b^{i+1} \ldots a^{n}b^{n} \ind^{u}_{M} a^{i}b^{i}$ for $0 \leq i \leq n$, and therefore

$$cb^{i+1} \ldots b^{n} \ind_{M}^{u} b^{i}$$ for $0 \leq i < n$.

Now for $1 \leq i \leq n$ we have $ca^{i}b^{i} \equiv_{M} cab$. Let $e_{i}ca^{i}b^{i} \equiv_{M} ecab$ for $1 \leq i \leq n$. Then

$$\bigwedge\alpha(x, e_{i})\vdash^{\mathcal{I}} \psi(x, c) \vee \bigvee_{i=1}^{n}\varphi(x, b^{i}) \vee \bigwedge_{i=1}^{n} \varphi_{N}(x, a^{i})$$ But by choice of the $a^{i}$,

$$\bigwedge\alpha(x, e_{i})\vdash^{\mathcal{I}} \psi(x, c) \vee \bigvee_{i=1}^{n}\varphi(x, b^{i})$$ Now for $1 \leq i \leq n$, because $b^{i} \equiv_{M} b$, $\varphi(x, b^{i})$ will be of the form $\bigvee_{j=1}^{N-1}\varphi_{j}(x, a'_{j})$ for $\varphi_{j}(x, a'_{j})$ $\mathcal{I}$-Kim-dividing over $M$. So, as the first of $n$ steps, we can apply $cb^{2} \ldots b^{n} \ind^{u} b^{1}$ and the inductive hypothesis on $N$ to find some conjunction $\beta(x, \overline{e})$ of conjugates of $\bigwedge\alpha(x, e_{i})$ (which will therefore be a conjunction of conjugates of $\alpha(x, e)$) so that 

$$\beta(x, \overline{e})\vdash^{\mathcal{I}} \psi(x, c) \vee \bigvee_{i=2}^{n}\varphi(x, b^{i})$$ Repeating $n-1$ more times, we are done.

\end{proof}

We now begin our construction. The following terminology comes from the notion of \text{strong finite character} (used in e.g. \cite{CR15}).

\begin{definition}
Let $\ind$ be an invariant relation between sets over a model. We say that $\ind$ satisfies \emph{quasi-strong finite character} if for $p, q$ complete types over some model $M$, $\{a, b \models p(x) \cup q(y): a\ind_{M} b\}$ is type-definable.
\end{definition}

\begin{definition}
Let $\ind$ be an invariant relation between sets over a model satisfying monotonicity, right extension and quasi-strong finite character, and fix a complete type $P(x)$ over a model $M$.

(1) A set of formulas $\{\varphi_{i}(x, b_{i})\}_{i \in I}$ is \emph{$h^{\ind}$-inconsistent with respect to $P(x)$} if there is no $a \models P(x)$ with $a \ind_{M} \{b_{i}\}_{i \in I}$ and $\models \varphi_{i}(a, b_{i})$ for all $i \in I$.

(2) A formula $\varphi(x, b)$ \emph{$h^{\ind}$-Kim-divides with respect to $P(x)$} if there is a coheir Morley sequence $\{b_{i}\}_{i \in \omega}$ starting with $b$ so that $\{\varphi(x, b_{i})\}_{i \in \omega}$ is $h^{\ind}$-inconsistent with respect to $P(x)$.

(3) A formula \emph{$h^{\ind}$-Kim-forks} with respect to $P(x)$ if it implies a disjunction of formulas $h^{\ind}$-Kim-dividing with respect to $P(x)$.

(4) A formula $\varphi(x, b)$ \emph{$h^{\ind}$-quasi-divides} over $M$ with respect to $P(x)$ if there are $b_{1}, \ldots, b_{n}$ with $b_{i} \equiv_{M} b$ and $\{\varphi(x, b_{i})\}_{i=1}^{n}$ $h^{\ind}$-inconsistent with respect to $P(x)$.
\end{definition}

\begin{lemma} \label{3-hinc}
(1) The sets defined by formulas $\varphi(x, b)$ so that $\{\varphi(x, b)\}$ is $h^{\ind}$-inconsistent with respect to $P(x)$ form an $M$-invariant ideal $\mathcal{I}^{\ind}_{P(x)}$.

(2) A set $\{\varphi_{i}(x, b_{i})\}_{i \in I}$ is $h^{\ind}$-inconsistent with respect to $P(x)$ if and only if some finite subset is (so its conjunction defines a set in the ideal $\mathcal{I}^{\ind}_{P(x)}$.)
\end{lemma}

\begin{proof}
For (1), it suffices to show (a) that if $\models\forall x(\varphi(x, b) \rightarrow \psi(x, c)) $, and $\psi(x, c)$ is $h^{\ind}$-inconsistent with respect to $P(x)$, then $\varphi(x, b)$ is $h^{\ind}$-inconsistent with respect to $P(x)$, and (b) that if both $\varphi(x, b)$ and $\psi(x, c)$ are $h^{\ind}$-inconsistent with respect to $P(x)$ then so is $\varphi(x, b) \vee \psi(x, c)$. For (a), suppose otherwise; then there is some realization $a$ of $P(x)$ with $\models \varphi(a, b)$ and $a \ind_{M} b$. By right extension, we can assume $a \ind_{M} bc$. But then $\models \psi(a, c)$, and by right monotonicity, $a \ind_{M} c$, contradicting that $\psi(x, c)$ is $h^{\ind}$-inconsistent with respect to $P(x)$. For (b), suppose otherwise; then there is some realization $a$ of $P(x)$ with $\models \varphi(a, b) \vee \psi(a, c)$ and $a \ind_{M} bc$; without loss of generality, $\models \varphi(a, b)$, and by right monotonicity, $a \ind_{M} b$, contradicting that $\varphi(x, b)$ is $h^{\ind}$-inconsistent with respect to $P(x)$. The proof of (a) also gives us the fact that a set $\{\varphi_{i}(x, b_{i})\}_{i \in I}$ is $h^{\ind}$-inconsistent with respect to $P(x)$ if some finite subset is (so its conjunction defines a set in the ideal $\mathcal{I}^{\ind}_{P(x)}$). To complete (2), we show the ``only if" direction. If $\{\varphi_{i}(x, b_{i})\}_{i \in I}$ is $h^{\ind}$-inconsistent with respect to $P(x)$ then there is no realization $a$ of $P(x) \cup \{\varphi_{i}(x, b_{i})\}_{i \in I}$ with $a \ind_{M} \{b_{i}\}_{i \in I}$. But the set of realizations $a$ of $P(x)$ that satisfy $a \ind_{M} \{b_{i}\}_{i \in I}$ is, by quasi-strong finite character, type-definable. So by compactness, there must be some finite $I_{0} \subseteq I$ so there is no realization $a$ of $P(x) \cup \{\varphi_{i}(x, b_{i})\}_{i \in I_{0}}$ with $a \ind_{M} \{b_{i}\}_{i \in I}$. But if there is a realization $a$ of $P(x) \cup \{\varphi_{i}(x, b_{i})\}_{i \in I_{0}}$ with $a \ind_{M} \{b_{i}\}_{i \in I_{0}}$, then we can even get $a \ind_{M} \{b_{i}\}_{i \in I}$ by right-extension, so $\{\varphi_{i}(x, b_{i})\}_{i \in I_{0}}$ will be as desired.
\end{proof}

\begin{corollary}\label{3-quasidividing}
For all formulas, $h^{\ind}$-Kim-forking with respect to $P(x)$ implies $h^{\ind}$-quasi-dividing with respect to $P(x)$.
\end{corollary}

\begin{proof}
By Lemma \ref{3-hinc}, $h^{\ind}$-Kim-dividing with respect to $P(x)$ is just $\mathcal{I}^{\ind}_{P(x)}$-Kim-dividing. Apply Lemma \ref{3-broom} to $\mathcal{I}^{\ind}_{P(x)}$.
\end{proof}

\begin{lemma}\label{3-lmon}
If a formula $\varphi(x, b)$ is $h^{\ind}$-inconsistent with respect to $P(x)$, then it is $h^{\ind}$-inconsistent with respect to any complete type $Q(x, y)$ extending $P(x)$. So the same is true for $h^{\ind}$-Kim-dividing and $h^{\ind}$-Kim-forking.
\end{lemma}

\begin{proof}
Suppose otherwise. Then there is a realization $ac$ of $Q(x, y) \cup \{\varphi(x, b)\}$ with $ac \ind_{M} b $. So by left monotonicity, $a \ind_{M} b $, but $a$ realizes $P(x) \cup \{\varphi(x, b)\}$, a contradiction.
\end{proof}

We are now in a position to study derived independence relations:

\begin{definition}
Let $\ind$ be an invariant relation between sets over a model satisfying monotonicity, right extension and quasi-strong finite character. Then we define $a\ind'_{M}b$ to mean that $\mathrm{tp}(a/Mb)$ does not contain any formulas $h^{\ind}$-Kim-forking with respect to $\mathrm{tp}(a/M)$.
\end{definition}

\begin{lemma}\label{3-derived}
Suppose $\ind$ is an invariant relation between sets over a model satisfying monotonicity, right extension, quasi-strong finite character, and full existence. Then so is $\ind'$.
\end{lemma}

\begin{proof}
Invariance is obviously inherited from $\ind$. Quasi-strong finite character is by construction and right extension is also standard from the construction: if $a \ind'_{M}b$ but, for some $c \in \mathbb{M}$ there is no $a' \equiv_{Mb} a$ with $a \ind'_{M}bc$, then $\mathrm{tp}(a/Mb)$ must imply a disjunction of formulas with parameters in $Mbc$ $h^{\ind}$-Kim-forking with respect to $P(x)$; some formula in $\mathrm{tp}(a/Mb)$ must then imply this disjunction, which will then $h^{\ind}$-Kim-fork with respect to $P(x)$, contradicting $a \ind'_{M}b$. Right monotonicity is by definition. Left monotonicity is Lemma \ref{3-lmon}. It remains to show full existence; the proof is a straightforward generalization of the proof of Lemma 3.7 of \cite{CK09}. By right extension, it suffices to show that $b \ind'_{M} M$ for any tuple $b$ (the ``existence" property that is implied by full existence). Suppose otherwise; then $\mathrm{tp}(b/M)$ contains a formula $\varphi(x, m)$ for $m \in M$ that $h^{\ind}$-Kim-forks over $M$. By Corollary \ref{3-quasidividing}, $\varphi(x, m)$ $h^{\ind}$-quasi-divides over $M$. Since $m \in M$, this just means that $\varphi(x, m) \in \mathcal{I}_{\mathrm{tp}(b/M)}^{\ind}$. But since $\varphi(x, m) \in \mathrm{tp}(b/M)$, this contradicts full existence for $\ind$.
\end{proof}

The next observation is required to produce a relation with the coheir chain condition:

\begin{lemma}\label{3-chain}
Let $\ind$ be as in Lemma \ref{3-derived} and suppose $a \ind'_{M} b$. Then for $I=\{b_{i}\}_{i \in \omega}$ a coheir Morley sequence starting with $b$, there is $I' \equiv_{M} I$ with $a \ind_{M} I'$ and each term of $I'$ satisfying $\mathrm{tp}(b/Ma)$. In particular, $\ind'$ implies $\ind$, so $h^{\ind}$-Kim-forking implies $h^{\ind'}$-Kim-forking.
\end{lemma}

\begin{proof}
Suppose otherwise: then for $q= \mathrm{tp}(a, b/M)$, $\cup_{i\in \omega}q(x,b_{i})$ is $h^{\ind}$-inconsistent with respect to $\mathrm{tp}(a/M)$, so by part (2) of Lemma \ref{3-hinc}, some finite subset must be $h^{\ind}$-inconsistent with respect to $\mathrm{tp}(a/M)$. This gives us a formula in $q(x, b)$ that $h^{\ind}$-Kim divides with respect to $\mathrm{tp}(a/M)$, a contradiction.
\end{proof}

Note that $\ind^{h}$ satisfies the assumptions of Lemma \ref{3-derived}. Now define inductively, $\ind^{(0)}=\ind^{h}$, $\ind^{(n+1)}=(\ind^{(n)})^{'}$. Let $\ind^{\mathrm{CK}}= \bigcap_{i=0}^{\infty}\ind^{(n)}$. Then because $h^{\ind^{(n)}}$-Kim-forking implies $h^{\ind^{(n+1)}}$-Kim-forking, and $a \ind^{\mathrm{CK}}_{M} b$ means that $\mathrm{tp}(a/Mb)$ does not contain a $h^{\ind^{(n)}}$-Kim-forking formula for any $n$, right extension and quasi-strong finite character are standard. Monotonicity and invariance follows from monotonicity and invariance of the $\ind^{(n)}$. By right extension for $\ind^{\mathrm{CK}}$, full existence for $\ind^{\mathrm{CK}}$ would follow from the existence property $b \ind^{\mathrm{CK}}_{M} M$ for any $b$, but this just follows from full existence for each of the $\ind^{(n)}$. Finally, the coheir chain condition follows from Lemma \ref{3-chain} together with quasi-strong finite character for the $\ind^{(n)}$ and compactness.

It remains to show that $\ind^{\mathrm{CK}}$ is the weakest relation implying $\ind^{h}$ and satisfying these properties. Let $\ind$ be some other such relation and assume by induction that $\ind $ implies $\ind^{(n)}$ Assume $a \ind_{M} b$; we show $a \ind^{(n+1)}_{M} b$. Suppose otherwise; by right extension for $\ind$, we can assume $\mathrm{tp}(a/Mb)$ contains a formula $\varphi(x, b)$ that $h^{\ind^{(n)}}$-Kim-divides with respect to $\mathrm{tp}(a/M)$. Let $I=\{b_{i}\}_{i \in \omega}$ be a coheir Morley sequence starting with $b$ witnessing this. Then by the coheir chain condition for $\ind$, there is some $a'$ with $a' \ind_{M} I$, so in particular $a' \ind^{(n)}_{M} I$ by induction, and with $a'b_{i} \equiv_{M} ab$ for $ i \in \omega $, so in particular with $a'$ satisfying $\{\varphi(x, b_{i})\}_{i \in \omega}$, a contradiction. 

This completes the proof of Proposition \ref{3-ck}.

\begin{remark}
If $\mathbb{M}' \succ \mathbb{M}$ is a very large (sufficiently saturated) model, then $\ind^{\mathrm{CK}}$ as computed in $\mathbb{M}'$ restricts to $\ind^{\mathrm{CK}}$ as computed in $\mathbb{M}$. We can see that $\ind^{\mathrm{CK}}$ has this property as it is true for $\ind^{h}$ and is preserved by going from $\ind^{(n)}$ to $\ind^{(n+1)}$. However, it is also immediate that invariance, monotonicity, full existence and right extension, and the coheir chain condition are preserved on restriction.
\end{remark}

\section{Canonical coheirs in $\mathrm{NSOP}_{2}$ theories}

The goal of this section is to prove a version of ``Kim's lemma for Kim-dividing" for canonical Morley sequences in $\mathrm{NSOP}_{2}$ theories.

\begin{lemma}\label{3-ccoheir}
Let $p(x)$ be a type over $M$. Then it has a global extension $q(x)$ so that for all tuples $b \in \mathbb{M}$, if $c \models q|Mb$, then $b \ind^{\mathrm{CK}}_{M}c$. So in particular, $q$ is a global coheir of $p(x)$.
\end{lemma}

\begin{proof}
In a very large $\mathbb{M}' \succ \mathbb{M}$, full existence and invariance for $\ind^{\mathrm{CK}}$, and an automorphism, gives us a realization $c'$ of $p(x)$ with $\mathbb{M} \ind^{\mathrm{CK}}_{M} c'$. Now take $q(x)$ to be $\mathrm{tp}(c'/
\mathbb{M})$, and the lemma follows by monotonicity on the left.
\end{proof}

\begin{definition}
We call $q(x)$ as in Lemma \ref{3-ccoheir} a \emph{canonical coheir}, and a coheir Morley sequence in it a \emph{canonical Morley sequence}.
\end{definition}

\begin{theorem}\label{3-cms}
Let $T$ be $\mathrm{NSOP}_{2}$. Suppose a canonical Morley sequence witnesses Kim-dividing of a formula $\varphi(x, b)$ over $M$. Then there is a finite bound (depending only on $\varphi(x, y)$ and the degree of Kim-dividing witnessed by the canonical Morley sequence) on the length of a coheir sequence $\{b_{i}\}_{i=1}^{n}$ over $M$ of realizations of $\mathrm{tp}(b/M)$ so that $\{\varphi(x, b_{i})\}_{i =1}^{n}$ is consistent. In particular, every coheir Morley sequence starting with $b$ witnesses Kim-dividing of $\varphi(x, b)$ over $M$.
\end{theorem}

To start, we introduce the notion of a \textit{coheir tree} in a general theory $T$.
\begin{definition}\label{3-ct2}
Let $p$ be any type over $M$. We say that a tree $(b_{\eta})_{\eta \in \omega^{\leq n}}$ of realizations of $p$ is a \emph{coheir tree} in $p$ if

(1) for each $\mu \in \omega^{< n}$, $(\{b_{\eta}\}_{\eta \unrhd \mu \smallfrown\langle i \rangle })_{i= 0}^{\infty}$ (the sequence consisting of the subtrees above a fixed node) is a coheir Morley sequence over $M$.

(2) there are global coheir extensions $q_{0}, \ldots, q_{n}$ of $p$ so that for each $\mu \in \omega^{n-m}$, $b_{\mu} \models q_{m}|_{\{b_{\eta}\}_{\eta \rhd \mu}}$.
\end{definition}

The key lemma of this section allows us to construct coheir trees in any theory so that sequences of nodes with common meet are canonical Morley sequences. Abusing the language by nodes, paths, etc. we often refer to the tuples which they index; the term ``descending comb" will have a similar meaning in a tree of finite height or a set of subtrees as it does in $\omega^{<\omega}$.

\begin{lemma}\label{3-ct}
Let $p(x)$ be any type over $M$. Let $q(x)$ be a canonical coheir extension of $p(x)$. Let $b_{0}, \ldots, b_{n}$ be a coheir sequence over $M$ of realizations of $p$. Then there is a coheir tree indexed by $\omega^{\leq n}$, any path of which, read in the direction of the root, realizes $\mathrm{tp}(b_{0} \ldots b_{n}/M)$, and descending comb of which, read in lexicographic order, begin a canonical Morley sequence in $q(x)$.
\end{lemma}

\begin{proof}
We need the following claim:

\begin{claim}\label{3-treecopy}
If $a \ind^{\mathrm{CK}}_{M}b$ and $I$ is a coheir tree in $\mathrm{tp}(b/M)$, then there is some $I' \equiv_{M} I$ with $a \ind^{\mathrm{CK}}_{M}I'$ (so in particular $I' \ind^{u}_{M} a $) each term of which satisfies $\mathrm{tp}(b/Ma)$.
\end{claim}

\begin{proof}
Let $I=\{b_{\eta}\}_{\eta \in \omega^{\leq n}}$; we find $I'=\{b'_{\eta}\}_{\eta \in \omega^{\leq n}}$ as desired. The proof is by downward induction on $k$: suppose $\{b'_{\eta}\}_{\eta \unrhd \zeta_{n-k}}$ is already constructed, and we construct $\{b'_{\eta}\}_{\eta \unrhd \zeta_{n-(k+1)}}$. First, $\{b'_{\eta}\}_{\eta \rhd \zeta_{n-(k+1)}}$ comes directly from the chain condition. Second, by left extension for $\ind^{u}$, find some copy $J$ of $\{b'_{\eta}\}_{\eta \rhd \zeta_{n-(k+1)}}$ over $M$ with $J \ind^{u}_{M}\{b'_{\eta}\}_{\eta \rhd \zeta_{n-(k+1)}}$ and some arbitrary term of $J$ satisfying the conjugate to $M\{b'_{\eta}\}_{\eta \rhd \zeta_{n-(k+1)}}$ of $\mathrm{tp}(b_{\zeta_{n-(k+1)}}/M\{b_{\eta}\}_{\eta \rhd \zeta_{n-(k+1)}})$ (that is, $q_{k
+1}(x)|_{M\{b'_{\eta}\}_{\eta \rhd \zeta_{n-(k+1)}}}$ from Definition \ref{3-ct2}). Then use the chain condition to find some $J' \equiv_{M\{b'_{\eta}\}_{\eta \rhd \zeta_{n-(k+1)}}}J$ with $J'\equiv_{Ma}\{b'_{\eta}\}_{\eta \rhd \zeta_{n-(k+1)}}$ and $a \ind^{\mathrm{CK}}_{M}\{b'_{\eta}\}_{\eta \rhd \zeta_{n-(k+1)}}J'$. Finally, using monotonicity on the right, discard all the terms of $J'$ other than the one corresponding to the chosen term of $J$, to obtain $b'_{\zeta_{n-(k+1)}}$.
\end{proof}

Now by induction, it suffices to show this for a coheir sequence $b_{0} \ldots, b_{n+1}$ assuming  $I_{n} = (b_{\eta})_{\eta \in \omega^{\leq n}}$ is already constructed for $b_{0} \ldots, b_{n}$. First, we find a long coheir sequence $\{I_{n}^{i}\}_{i=0}^{\alpha}$ of realizations of $\mathrm{tp}(I_{n}/M)$ so that each node of $I_{n}^{\gamma}$ satisfies $q(x)|_{M\{I_{n}^{i}\}_{i<\gamma}}$; then having taken it long enough, we can find a coheir Morley sequence $\{I_{n}^{i}\}_{i=0}^{\omega}$ with the same property, preserving the condition on descending combs. (Any descending comb inside of these copies will either lie inside of one copy of $I_{n}$, so will of course begin a descending Morley sequence inside of $q(x)$ by the induction hypothesis, or will consist of a descending comb inside one copy $I_{n}^{i}$ followed by an additional node of a later copy $I_{n}^{j}$ for $i<j$, which will indeed continue the Morley sequence in $q(x)$ begun by the previous nodes.) Suppose $\{I_{n}^{i}\}_{i<\gamma}$ already constructed; taking $a = \{I_{n}^{i}\}_{i<\gamma}$ in the above claim and $b \models q(x)|_{M\{I_{n}^{i}\}_{i<\gamma}}$, we can choose $I^{\gamma}_{n}$ to be the $I'$ given by the claim. 

Now let $q_{n+1}$ be a global extension, finitely satisfiable in $M$, of $\mathrm{tp}(b_{n+1}/Mb_{0}\ldots b_{n})$. Then we take $b\models q_{n+1}(x)|_{M\{I_{n}^{i}\}_{i=0}^{\infty}}$ as the new root, guaranteeing the condition on paths. Now reindex accordingly.

\end{proof}

We can now prove Theorem \ref{3-cms}. Let $q(x)$ be a canonical coheir extension of $\mathrm{tp}(b/M)$ and $k$ the degree of Kim-dividing for $\varphi(x, b)$ witnessed by a canonical Morley sequence in $q(x)$. Let $\{b_{i}\}_{i=0}^{n}$ be a coheir sequence over $M$ of realizations of $\mathrm{tp}(b/M)$ so that $\{\varphi(x, b_{i})\}_{i =0}^{n}$ is consistent. Then the coheir tree given by the previous lemma gives the first $n+1$ levels of an instance of $k$-$\mathrm{DCTP}_{1}$: the $k$-dividing witnessed by canonical Morley sequences in $q(x)$ gives the inconsistency condition for descending combs of size $k$, and the consistency of $\{\varphi(x, b_{i})\}_{i =0}^{n}$ gives the consistency of the paths. So if $n$ is without bound, we must have $k$-$\mathrm{DCTP}_{1}$ for $\varphi(x, y)$ by compactness, and thus $\mathrm{SOP}_{2}$ by lemma \ref{3-dctp1}. This concludes the proof of \ref{3-cms}.

We have some applications of this proof to a notion related to the $\mathrm{NATP}$ theories introduced by Ahn and Kim in \cite{AK20}, and studied in greater depth by Ahn, Kim and Lee in \cite{AK21}, assuming the $\mathrm{NATP}$ analogue of lemma \ref{3-dctp1}. The result for $\mathrm{NATP}$ theories would be interesting because while $\mathrm{NSOP}_{1}$ theories are $\mathrm{NATP}$ \cite{AK20}, as Ahn, Kim and Lee have shown in \cite{AK21}, there are examples of $\mathrm{NATP}$ $\mathrm{SOP}_{1}$ theories. The following is the original definition from \cite{AK20}:

\begin{definition}
The theory $T$ has \emph{$\mathrm{NATP}$ (the negation of the antichain tree property)} if there does not exist a formula $\varphi(x, y)$ and tuples $\{b_{\eta}\}_{\eta \in 2^{<\omega}}$ so that $\{\varphi(x, b_{\sigma \upharpoonleft n})\}_{n \in \omega}$ is $2$-inconsistent for any $\sigma \in 2^{\omega}$, but for pairwise incompararable $\eta_{1}, \ldots, \eta_{l} \in 2^{< \omega}$, $\{\varphi(x, b_{\eta_{i}})\}_{i=1}^{l}$ is consistent.
\end{definition}

In \cite{AK21}, Ahn, Kim and Lee define a theory to have $k$\textit{-}$\mathrm{ATP}$ if the above fails replacing $2$-inconsistency with $k$-inconsistency, and show that for any $k \geq 2$, a theory fails to be $\mathrm{NATP}$ (that is, has $2$-$\mathrm{ATP}$) if and only if it has $k$-$\mathrm{ATP}$. That is, they show the analogue for $\mathrm{NATP}$ theories of results of Kim and Kim in \cite{KK11} on $\mathrm{NSOP}_{2}$ theories, but of not those claimed by Chernikov and Ramsey in \cite{CR15}, nor of the above Lemma \ref{3-dctp1}. One might ask whether, for any $k$, the following definition is equivalent to the failure of $\mathrm{NATP}$:

\begin{definition}
The theory $T$ has \emph{$k$-$\mathrm{DCTP}_{2}$} if there exists a formula $\varphi(x, y)$ and tuples $\{b_{\eta}\}_{\eta \in 2^{<\omega}}$ so that $\{\varphi(x, b_{\sigma \upharpoonleft n})\}_{n \in \omega}$ is $k$-inconsistent for any $\sigma \in 2^{\omega}$, but for any descending comb $\eta_{1} \ldots, \eta_{l} \in 2^{< \omega}$, $\{\varphi(x, b_{\eta_{i}})\}_{i=1}^{l}$ is consistent.
\end{definition}

If so, then the following applies to $\mathrm{NATP}$ theories:

\begin{theorem}
Let $T$ be a theory so that, for all $k \geq 2$, $T$ does not have $k$-$\mathrm{DCTP}_{2}$. Let $M$ be any model and $b$ any tuple. Then there is a global type extending $\mathrm{tp}(b/M)$, finitely satisfiable in $M$, so that for any formula $\varphi(x, y)$ with parameters in $M$, if coheir Morley sequences in this type do not witness Kim-dividing of $\varphi(x, b)$, no coheir Morley sequence over $M$ starting with $b$ witnesses Kim-dividing of $\varphi(x, b)$ over $M$.
\end{theorem}

This follows from the same construction. The following corollary is standard; see Corollary 3.16 of \cite{CK09} for a similar argument:

\begin{corollary}
If, for all $k \geq 2$, $T$ does not have $k$-$\mathrm{DCTP}_{2}$, then Kim-forking (with respect to coheir Morley sequences) coincides with Kim-dividing (with respect to coheir Morley sequences).
\end{corollary}

\section{Conant-independence in $\mathrm{NSOP}_{2}$ theories}
We introduce a notion of independence which will generalize, in the proof of the main result of this paper, the role played by $\ind^{a}$ in the \textit{free amalgamation theories} introduced in \cite{Co15}. The notation $\ind^{K^{*}}$ comes from the related notion of Kim-independence from \cite{KR17}, $\ind^{K}$; a similar notion involving dividing with respect to \textit{all} (invariant) Morley sequences is suggested in tentative remarks of Kim in \cite{K09}.

\begin{definition}
Let $M$ be a model and $\varphi(x, b)$ a formula. We say $\varphi(x, b)$ \emph{Conant-divides} over $M$ if for \emph{every} coheir Morley sequence $\{b_{i}\}_{i \in \omega}$ over $M$ starting with $b$, $\{\varphi(x, b_i)\}_{i \in \omega}$ is inconsistent. We say $\varphi(x, b)$ \emph{Conant-forks} over $M$ if and only if it implies a disjunction of formulas Conant-dividing over $M$. We say $a$ is \emph{Conant-independent} from $b$ over $M$, written $a \ind^{K^{*}}_{M}b$, if $\mathrm{tp}(a/Mb)$ does not contain any formulas Conant-forking over $M$.
\end{definition}

Note that this definition differs from the standard definition of Conant-independence given in \cite{GFA}, in that it uses coheir Morley sequences rather than invariant Morley sequences. In \cite{KM22} Alex Kruckman and the author show how to carry out this proof with the standard Conant-independence. We may also dualize Theorem 3.10 of \cite{KL22}.

\begin{proposition}\label{3-cfd}
In any theory $T$, Conant-forking coincides with Conant-dividing for formulas, and $\ind^{K^{*}}$ has right extension.
\end{proposition}

\begin{proof}
We see first of all that Conant-dividing is preserved under adding and removing unused parameters: it suffices to show that if $\models \forall x \varphi(x, a) \leftrightarrow \varphi'(x, ab)$ then $\varphi(x, a)$ Conant-divides over $M$ if and only if $\varphi'(x, ab)$ Conant-divides over $M$. Let $\{a_{i}b_{i}\}_{i \in \omega}$ be a coheir Morley sequence starting with $ab$ witnessing the failure of Conant-dividing of the latter; then $\{a_{i}\}_{i \in \omega}$ witnesses the failure of Conant-dividing of the former. Conversely, let $\{a_{i}\}_{i \in \omega}$ be a coheir Morley sequence starting with $a$ witnessing the failure of Conant-dividing of $\varphi(x, a)$; then by Claim \ref{3-lext} and an automorphism there are $\{b_{i}\}_{i \in \omega}$ so that $\{a_{i}b_{i}\}_{i \in \omega}$ is a coheir Morley sequence starting with $ab$, and this will witness the failure of Conant-dividing of $\varphi'(x, ab)$. The result is now standard, following, say, the proof in \cite{KR17} of the analogous fact for Kim-dividing under Kim's lemma. Suppose $\varphi(x, b)$ Conant-forks over $M$ but does not Conant-divide over $M$; by the above we can assume it implies a disjunction of the form $\bigvee_{i=1}^{n} \varphi_{i}(x, b)$ where $\varphi_{i}(x, b)$ Conant-divides over $M$. Let $\{b_{i} \}_{i \in \omega}$ be a coheir Morley sequence starting with $b$ witnessing the failure of Conant-dividing, so there is some $a$ realizing $\{\varphi(x, b_{i})\}_{i \in \omega}$. Then by the pigeonhole principle, there is some $1 \leq k \leq n$ so that $a$ realizes infinitely many of the $\varphi_{k}(x, b_{i})$. By an automorphism this contradicts Conant-dividing of $\varphi_{k}(x, b).$

Right extension is standard and exactly as in Lemma \ref{3-derived}: if $a \ind^{K^{*}}_{M}b$ but there is no $a' \equiv_{Mb} a$ with $a' \ind^{K^{*}}_{M}bc$, then $\mathrm{tp}(a/Mb)$ must imply a disjunction of formulas with parameters in $Mbc$ Conant-forking over $M$; some formula in $\mathrm{tp}(a/Mb)$ must then imply this disjunction, which will then Conant-fork over $M$, contradicting $a \ind^{K^{*}}_{M}b$.
\end{proof}

The following is immediate from Theorem 4.3:

\begin{corollary}
Let $T$ be $\mathrm{NSOP}_{2}$. Then a formula Conant-divides (so Conant-forks) over $M$ if and only if it Kim-divides with respect to some (any) canonical Morley sequence.
\end{corollary}

We develop the theory of Conant-independence in $\mathrm{NSOP}_{2}$ theories in analogy with the theory of Kim-independence in $\mathrm{NSOP}_{1}$ theories.

\begin{proposition} \label{3-ccc}
(Canonical Chain Condition): Let $T$ be $\mathrm{NSOP}_{2}$ and suppose $a \ind^{K^{*}}_{M} b$. Then for any canonical Morley sequence $I$ starting with $b$, we can find some $I' \equiv_{Mb} I$ indiscernible over $a$; any such $I'$ will satisfy $a \ind^{K^{*}}_{M} I'$.
\end{proposition}

\begin{proof}
This is similar to the proof of, say, the analogous fact about Kim-independence in $\mathrm{NSOP}_{1}$ theories (Proposition 3.21 of \cite{KR17}). The existence of such an $I'$ follows from the previous corollary by Ramsey and compactness. To get $a \ind^{K^{*}}_{M} I'$, let $I'= \{b_{i}\}_{i \in \omega}$; it suffices to show $a \ind^{K^{*}}_{M} b_{0} \ldots b_{n-1}$ for any $n$. But $\{b_{in}b_{in+1} \ldots b_{in+(n-1)} \}_{i \in \omega}$ is a coheir Morley sequence over $M$ starting with $b_{0} \ldots b_{n-1}$, each term of which satisfies $\{b_{0}\ldots b_{n-1}/Ma\}$, so $a \ind^{K^{*}}_{M} b_{0} \ldots b_{n-1}$ follows.
\end{proof}

\begin{theorem}\label{3-symm}
Let $T$ be $\mathrm{NSOP}_{2}$. Then Conant-independence is symmetric.
\end{theorem}

\begin{proof}
Suppose otherwise, so for some $a, b \in \mathbb{M}$, $a \ind^{K^{*}}_{M} b$ but $b$ is Conant-dependent on $a$ over $M$. We use $a \ind^{K^{*}}_{M} b$ to build trees as in the proof of symmetry of Kim-independence for $\mathrm{NSOP}_{1}$ theories (the construction is Lemma 5.11 of \cite{KR17}.) Specifically, what we want is, for any $n$, a tree $(I_{n}, J_{n})= (\{a_{\eta}\}_{\eta \in \omega^{\leq n}}, \{b_{\sigma}\}_{\sigma \in \omega^{n}})$, infinitely branching at the first $n+1$ levels and then with each $a_{\sigma}$ for $\sigma \in \omega^{n}$ at level $n+1$ followed by a single additional leaf $b_{\sigma}$ at level $n+2$, satisfying the following properties:

(1) For $\eta \unlhd \sigma$, $a_{\eta}b_{\sigma} \equiv_{M} ab$

(2) For $\eta \in \omega^{< n}$, the subtrees above $\eta$ form a canonical coheir sequence indiscernible over $a_{\eta}$, so by Proposition \ref{3-ccc}, $a_{\eta}$ is Conant-independent over $M$ from those branches taken together.

Suppose $(I_{n}, J_{n})$ already constructed; we construct $(I_{n+1}, J_{n+1})$. We see that the root $a_{\emptyset}$ of $(I_{n}, J_{n})$ is Conant-independent from the rest of the tree, $(I_{n}, J_{n})^{*}$: for $n = 0$ this is just the assumption $a \ind^{K^{*}}_{M} b$, where we allow $a_{\emptyset}b_{\emptyset} = ab$, while for $n > 0$ this is (2). So by extension we find $a_{\emptyset}' \equiv_{M(I_{n}J_{n})^{*}} a_{\emptyset}$ (so guaranteeing (1)), to be the root of $(I_{n+1}, J_{n+1})$, with $a \ind^{K^{*}}_{M} I_{n}J_{n}$. Then by Proposition \ref{3-ccc}, find some canonical Morley sequence $\{(I_{n},J_{n})^{i}\}_{i \in \omega}$ starting with $(I_{n},J_{n})$ indiscernible over $Ma'_{\emptyset}$, guaranteeing (2), and reindex accordingly.

Now let $\varphi(x, a) \in \mathrm{tp}(b/Ma)$ (so $\varphi(x, y)$ is assumed to have parameters in $M$) witness the Conant-dependence of $b$ on $a$ over $M$ and let $k$ be the (strict) bound supplied by Theorem 4.3. We show $I_{n}$ gives the first $n+1$ levels of an instance of $k$-$\mathrm{DCTP}_{1}$ for $\varphi(x, y)$, giving a contradiction to $\mathrm{NSOP}_{2}$ by compactness and lemma \ref{3-dctp1}. Consistency of the paths comes from (1). As for the inconsistency of a descending comb of size $k$, it follows from (2) (and the same reasoning as in the proof of Lemma \ref{3-ct}) that a descending comb forms a coheir sequence, so the inconsistency follows by choice of $k$.
\end{proof}

Note that by constructing a tree of size $\kappa$ and using an Erdős-Rado version of fact \ref{3-tt} (see Lemma 5.10 of \cite{KR17} for a result of this kind for similar kind of indiscernible tree, itself based on Theorem 1.13 of \cite{GIL02}), we could have assumed the tree we constructed in the above proof to be \textit{strongly indiscernible}. It follows that we could have only used that if a canonical Morley sequence witnesses Kim-dividing of a formula, then so does any coheir Morley sequence; the statement of Theorem 4.3 is somewhat stronger. (In fact, by using a local version of the chain condition--if $a \ind^{K^{*}}_{M} b$ and $\models \varphi(a, b)$, then there is some coheir Morley sequence $I=\{b_{i}\}_{i \in \omega}$ so that $b_{i} \equiv_{M} b$, $\models \varphi(a, b_{i})$ for $i \in \omega$, and $a \ind^{K^{*}}_{M}I$--we could have avoided Theorem 4.3 altogether up to this point, but we have not yet found a suitable replacement for the below ``weak independence theorem" that does not require it. We leave the details to the reader.)

We next aim to prove a version of the ``weak independence theorem." To formulate this, we need the following strengthening of Lemma \ref{3-ccoheir}:

\begin{lemma}\label{3-sccoheir}
Let $p(x)$ be a type over $M$. Then there is some global extension $q(x)$ of $p(x)$ so that, for all tuples $b \in \mathbb{M}$ if $c \in \mathbb{M}$ with $c \models q(x)|_{Mb}$, then for any $a \in \mathbb{M}$ there is $a' \equiv_{Mc}a$ with $a' \in \mathbb{M}$ so that $\mathrm{tp}(a'c/Mb)$ extends to a canonical coheir of $\mathrm{tp}(a'c/M)=\mathrm{tp}(ac/M)$. So in particular, $q(x)$ is a canonical coheir of $p(x)$.
\end{lemma}

\begin{proof}
Working again in a very large $\mathbb{M}' \succ \mathbb{M}$, find $\mathbb{M}_{1} \equiv_{M} \mathbb{M}$ with $\mathbb{M} \ind^{\mathrm{CK}}_{M} \mathbb{M}_{1}$ using full existence for $\mathrm{\ind^{CK}}$. Find a realization $c''$ of $p(x)$ in $\mathbb{M}_{1}$ and let $q(x)$ be its type over $\mathbb{M}$. Now suppose $b \in \mathbb{M}$ and $c \in \mathbb{M}$ with $c \models q(x)|_{Mb}$, and let $a \in \mathbb{M}$.  Then there is some $a'' \in \mathbb{M}_{1}$ with $a''c'' \equiv_{M} ac$. Because $c'' \equiv_{Mb} c$, there is some $a' \in \mathbb{M}$ with $a''c'' \equiv_{Mb} a'c$. Together with $a''c'' \equiv_{M} ac$, it follows that $a' \equiv_{Mc} a$. And $\mathrm{tp}(a'c/Mb)$ extends to $\mathrm{tp}(a''c''/\mathbb{M})$, which it remains to show is canonical. But by right monotonicity, $\mathbb{M} \ind^{\mathrm{CK}}_{M} a''c''$, so the result follows by left monotonicity (see also the proof of Lemma \ref{3-ccoheir}).
\end{proof}

\begin{definition}
We call $q(x)$ as in Lemma \ref{3-sccoheir} a \emph{strong canonical coheir}, and a coheir Morley sequence in it a \emph{strong canonical Morley sequence}.
\end{definition}

The proof of the following is as in Proposition 6.10 of \cite{KR17}:

\begin{proposition} \label{3-wit}
(Weak Independence Theorem) Assume $T$ is $\mathrm{NSOP}_{2}$. Let $a_{1} \ind^{K^{*}}_{M} b_{1}$, $a_{2} \ind^{K^{*}}_{M} b_{2}$, $a_{1} \equiv_{M} a_{2}$, and $\mathrm{tp}(b_{2}/Mb_{1})$ extends to a strong canonical coheir $q(x)$ of $\mathrm{tp}(b_{2}/M)$. Then there exists a realization $a$ of $\mathrm{tp}(a_{1}/Mb_{1}) \cup \mathrm{tp}(a_{2}/Mb_{2})$ with $a \ind^{K^{*}}_{M}b_{1}b_{2}$.
\end{proposition}

\begin{proof}
We start with the following claim, proven exactly as in \cite{KR17} but with canonical rather than invariant Morley sequences:

\begin{claim}\label{3-witclaim}
There exists some $b'_{2}$ with $a_{1}b'_{2} \equiv_{M}a_{2}b_{2}$ and $a_{1} \ind^{K^{*}}_{M}b_{1}b'_{2} $.
\end{claim}

\begin{proof}
It is enough by symmetry of $\ind^{K^{*}}$ to find $b'_{2}$ with $a_{1}b'_{2} \equiv_{M}a_{2}b_{2}$ and $b_{1}b'_{2} \ind^{K^{*}}_{M} a_{1} $. If $p(x, a_{2}) = \mathrm{tp}(b_{2} / Ma_{2})$ (leaving implied, throughout the proof of this claim, any parameters in $M$ in types and formulas), then by $a_{2} \ind^{K^{*}}_{M} b_{2}$ and symmetry we have $b_{2} \ind^{K^{*}}_{M} a_{2}$, so because $a_{1} \equiv_{M} a_{2}$ we know that $p(x, a_{1})$ contains no formulas Conant-forking over $M$. It suffices to show consistency of

$$p(x, a_{1}) \cup \{\neg\varphi(x, b_{1}, a_{1}): \varphi(x, y, a_{1}) \text{ Conant-forks over } M \}$$

Otherwise, by compactness and equivalence of Conant-forking with Conant-dividing, we must have $p(x, a_{1}) \vdash \varphi(x, b_{1}, a_{1})$ for some $\varphi(x, y, z)$ with $\varphi(x, y, a_{1})$ Conant-dividing over $M$. By symmetry, $b_{1} \ind_{M}^{K^{*}}a_{1}$. So Proposition \ref{3-ccc} yields a canonical Morley sequence $\{a_{1}^{i}\}_{i \in \omega}$ starting with $a_{1}$ and indiscernible over $Mb_{1}$. So 

$$\bigcup_{i =0}^{\omega} p(x, a_{1}^{i}) \vdash \{\varphi(x, b_{1}, a^{i}_{1})\}_{i \in \omega}$$ But because $p(x, a_{1})$ contains no formulas Conant-dividing over $M$ and $\{a_{1}^{i}\}_{i \in \omega}$ is a canonical Morley sequence, $\bigcup_{i =0}^{\infty} p(x, a_{1}^{i})$ is consistent, so $\{\varphi(x, b_{1}, a^{i}_{1})\}_{i \in \omega}$ and therefore $\{\varphi(x, y, a^{i}_{1})\}_{i \in \omega}$ is consistent. But this contradicts the fact that  $\varphi(x, y, a_{1})$ Conant-divides over $M$.
\end{proof}

We now complete the proof of the proposition. Let $p_{2}(x, b_{2}) = \mathrm{tp}(a_{2}/Mb_{2})$ (with parameters in $M$ left implied); we have to show that $\mathrm{tp}(a_{1}/Mb_{1})  \cup p_{2}(x, b_{2})$ has a realization $a$ with $a \ind^{K^{*}}_{M}b_{1}b_{2}$. So for $b''_{2} \equiv_{Mb_{1}} b_{2}$ with $b''_{2} \models q(x)|_{Mb_{1}b'_{2}}$, it suffices to show that $\mathrm{tp}(a_{1}/Mb_{1})  \cup p_{2}(x, b''_{2})$ has a realization $a$ with $a \ind^{K^{*}}_{M}b_{1}b''_{2}$. Using $b''_{2} \equiv_{M} b_{2} \equiv_{M} b'_{2}$, we find $b'_{1}$ with $b'_{1}b''_{2} \equiv_{M} b_{1}b'_{2}$; using Lemma \ref{3-sccoheir}, we can assume $\mathrm{tp}(b'_{1}b''_{2}/Mb_{1}b'_{2})$ extends to a canonical coheir of its restriction to $M$. So $b'_{1}b''_{2}, b_{1}b'_{2}$ begins a canonical Morley sequence $I$ over $M$, and by Proposition \ref{3-ccc} and an automorphism, there is some $a \equiv_{Mb_{1}b'_{2}} a_{1}$ with $a \ind^{K^{*}}_{M}I$ and therefore $a \ind^{K^{*}}_{M}b_{1}b''_{2}$, and with $I$ indiscernible over $Ma$. By $a \equiv_{Mb_{1}} a_{1}$ we have that $a$ realizes $\mathrm{tp}(a_{1}/Mb_{1})$, and by $ab''_{2}\equiv_{M} ab'_{2} \equiv_{M} a_{1}b'_{2} \equiv_{M} a_{2}b_{2}$ we have that $a$ realizes $p_{2}(x, b''_{2})$.

\end{proof}

\section{$\mathrm{NSOP}_{2}$ and $\mathrm{NSOP}_{1}$ theories}

We are now ready to prove that if $T$ is $\mathrm{NSOP}_{2}$, it is $\mathrm{NSOP}_{1}$. The proof follows Conant's proof (Theorem 7.17 of \cite{Co15}) that certain \emph{free amalgamation} theories are either simple or $\mathrm{SOP}_{3}$. As anticipated in Section 5, $\ind^{K^{*}}$ will play the role of $\ind^{a}$, while (strong) canonical Morley sequences will play the role of Morley sequences in the free amalgamation relation. This makes sense, as Lemma 7.6 of \cite{Co15} shows that $\ind^{a}$ is just Kim-independence with respect to Morley sequences in the free amalgamation relation, while Conant-independence in a $\mathrm{NSOP}_{2}$ theory is Kim-independence with respect to canonical Morley sequences. Similarly to how Conant uses free amalgamation and $\ind^{a}$ to show that a (modular) free amalgamation theory is either simple or $\mathrm{SOP}_{3}$, we will show by strong canonical types and $\ind^{K^{*}}$ that if $T$ is $\mathrm{NSOP}_{2}$, then

\:

$T$ is either $\mathrm{NSOP}_{1}$ or $\mathrm{SOP}_{3}$

\:

and therefore must be $\mathrm{NSOP}_{1}$. (In \cite{GFA}, we generalize Conant's work by studying abstract independence relations in potentially strictly $\mathrm{NSOP}_{1}$ or $\mathrm{SOP}_{3}$ theories, finding a more general set of axioms for these relations than Conant's free amalgamation axioms under which the $\mathrm{NSOP}_{1}$-$\mathrm{SOP}_{3}$ dichotomy holds and showing relationships with Conant-independence for \textit{invariant} rather than coheir Morley sequences--note that in Conant's free amalgamation theories, this is just $\ind^{a}$.)

We begin our proof.

Assume $T$ is $\mathrm{NSOP}_{2}$ and suppose $T$ is $\mathrm{SOP}_{1}$. Obviously Kim-dividing independence, $\ind^{Kd}$, implies $\ind^{K^{*}}$; the reverse implication would imply that $\ind^{Kd}$ is symmetric, contradicting $\mathrm{SOP}_{1}$ by Fact \ref{3-kimsymm}. So there are $a \ind^{K^{*}}_{M} b$ with $a$ Kim-dividing dependent on $b$ over $M$; let $r(x, y) = \mathrm{tp}(a, b/M)$, and let $\{b_{i}\}_{i \in \mathbb{N}}$ be a coheir Morley sequence over $M$ starting with $b$ such that $\{r(x, b_{i})\}_{i \in \omega}$ is $k$-inconsistent for some $k$. The following corresponds to Claim 1 of the proof of Theorem 7.17 in \cite{Co15}, but requires a different argument; see also \cite{Lee22} and footnote 1 of \cite{INDNSOP3}, for another argument involving the proof of Proposition 3.14 of \cite{KR17}:

\begin{claim}\label{3-2kd}
We can assume $k = 2$. More precisely, there are $\tilde{a}, \tilde{b} \in \mathbb{M}$ with $\tilde{a} \ind^{K^{*}}_{M}\tilde{b}$ and some coheir Morley sequence $\{\tilde{b}_{i}\}_{i \in \mathbb{N}}$ over $M$ starting with $\tilde{b}$ such that, for $\tilde{r}(\tilde{x}, \tilde{y})=:\mathrm{tp}(\tilde{a}, \tilde{b}/M)$, $\{\tilde{r}(x, \tilde{b}_{i})\}_{i \in \omega}$ is $2$-inconsistent.
\end{claim}

\begin{proof}
In particular there is no realization $a'$ of $\{r(x, b_{i})\}_{i < k}$ with $a' \ind^{K^{*}}_{M} b_{0} \ldots b_{k-1}$. Let $k^{*}$ be the maximal value of $k$ without this property, and $\tilde{b} = b_{0} \ldots b_{k^{*}-1}$. Then $\{\tilde{b}_{i}\}_{i \in \omega}= \{b_{ik^{*}} \ldots b_{ik^{*}+ k^{*} -1}\}_{i \in \mathbb{N}}$ is a coheir Morley sequence starting with $\overline{b}$. Let $a' \ind^{K^{*}}_{M} \overline{b}$ realize $\{r(x, b_{i})\}_{i < k^{*}}$, and let $r'(x, y)= \mathrm{tp}(a', \tilde{b}/M)$. Then by maximality and symmetry, there is no realization $a''$ of $r'(x, \tilde{b}_{0}) \cup r'(x, \tilde{b}_{1})$ with $\tilde{b}_{0}\tilde{b}_{1} \ind^{K^{*}}_{M} a''$. So there is no coheir Morley sequence $\{a'_{i}\}_{i \in \mathbb{N}}$ starting with $a'$, every term of which realizes $r'(x, \tilde{b}_{0}) \cup r'(x, \tilde{b}_{1})$. But by $a' \ind^{K^{*}}_{M} \tilde{b}$, symmetry and Proposition \ref{3-ccc}, there is some $M\tilde{b}$-indiscernible canonical Morley sequence $I$ starting with $a$ so that $I \ind^{K^{*}}_{M} \tilde{b}$. So let $\tilde{a}$ be $I$ and $\tilde{b}$ with $\overline{b}$. Since $\tilde{r}(\tilde{x}, \tilde{b}) =\mathrm{tp}(I / M\tilde{b})$ contains $\cup^{n}_{i=1}r'(x_{i}, \tilde{b})$, $\tilde{a}$ and $\tilde{b}$ are as desired.
\end{proof}

Now replace $a$ with $\tilde{a}$ and $b$ with $\tilde{b}$, as in claim \ref{3-2kd}; let $\rho(x, y) \in r(x, y) = \mathrm{tp}(a, b/ M)$ be such that $\{r(x, b_{i})\}_{i \in \omega}$ is $2$-inconsistent, by compactness. We have $b_{1} \ind^{K^{*}}_{M}b_{0}$, in analogy with Claim 2 of the proof of Theorem 7.17 of \cite{Co15}, because $b_{1} \ind^{u}_{M}b_{0}$ and clearly $\ind^{u}$ implies $\ind^{K^{*}}$.

Fix a strong canonical coheir extension $q(x)$ of $p(x)=\mathrm{tp}(b/M)$. We wish to construct, by induction, a configuration $\{b^{1}_{i}b_{i}^{2}\}_{i \in \omega}$ with the following properties:

(1) For $J_{n}$ the sequence beginning with $b_{i}^{2}$ for $i < n$ and then continuing with $b_{i}^{1}$ for $i \geq n$, $J_{n}$ is a strong canonical Morley sequence in $q(x)$.

(2) For $i \leq j$, $b_{i}^{1} b_{j}^{2}\equiv_{M}  b_{0}b_{1}$

(3) $b^{0}_{1} \ldots b^{1}_{n} \ind^{K^{*}}_{M}b^{2}_{0} \ldots b^{2}_{n}$ for any $n \in \omega$.

Then by $a \ind^{K^{*}}_{M} b$ (1) gives consistent sequences of instances of $r(x, y)$, while (2) gives inconsistent pairs by claim \ref{3-2kd}, so we can get an instance of $\mathrm{SOP}_{3}$ from this configuration exactly as in the argument at the end of the proof of Theorem 7.17 in \cite{Co15}, which we will reproduce for the convenience of the reader.

We make repeated use of symmetry for $\ind^{K^{*}}$ throughout. Suppose $\{b^{1}_{i}b_{i}^{2}\}_{i \leq n}$ already constructed. We start by adding $b^{1}_{n+1}$, and then add $b^{2}_{n+1}$. If we take $b^{1}_{n+1} \models q(x)|_{M b^{1}_{0}b^{2}_{0} \ldots b^{1}_{n}b^{2}_{n}}$ then (1) and (2) are preserved up to this point, and (3) is preserved by the following claim (which also holds of Kim-independence in $\mathrm{NSOP}_{1}$ theories):

\begin{claim}\label{3-invariance}
If $a \ind^{K^{*}}_{M} b$ and $\mathrm{tp}(c/Mab)$ extends to an $M$-invariant type $q(x)$, then $ac \ind^{K^{*}}_{M} b$.
\end{claim}

\begin{proof}
By Proposition \ref{3-ccc}, let $I=\{b_{i}\}_{i< \omega}$ be an $Ma$-indiscernible canonical Morley sequence over $M$ starting with $b$. Choose $c^{*}\models q|_{MIa}$, so for $i < \omega$, $b_{i} a \equiv_{Mc^{*}} b_{0} a = ba$. Since $I=\{b_{i}\}_{i< \omega}$ form a coheir Morley sequence with $b_{i} \equiv_{Mac^{*}} b$ for $i < \omega$, $ac^{*} \ind^{K^{*}}_{M} b$ by \ref{3-cfd}, so $ac \ind^{K^{*}} b $ as $c^{*} \models \mathrm{tp}(c/ab)$.

\end{proof}

Now by $b_{1} \ind^{K^{*}}_{M} b_{0}$ and the fact that $J_{0}$ is still a (strong) canonical Morley sequence up to this point, we can find a realization $b_{*} \ind^{K^{*}}_{M} b^{1}_{0} \ldots b^{1}_{n+1}$ of $\{t(b^{1}_{i}, y)\}_{i=1}^{n+1}$ for $t(x, y) = \mathrm{tp}(b_{0}b_{1}/M)$ by Proposition \ref{3-ccc} and an automorphism. Take $b^{*} \models q(x)|_{M b^{2}_{0} \ldots b^{2}_{n}}$, so $b^{*} \equiv_{M} b_{*}$; then this together with (3) allows us to apply Proposition \ref{3-wit} to the conjugate $q_{1}$ of $\mathrm{tp}(b^{1}_{0} \ldots b^{1}_{n+1}/Mb_{*})$ under an automorphism taking $b_{*}$ to $b^{*}$, and $q_{2}=\mathrm{tp}(b^{1}_{0} \ldots b^{1}_{n+1}/Mb^{2}_{0} \ldots b^{2}_{n})$. This and an automorphism (over $b^{2}_{0} \ldots b^{2}_{n}$, taking the Conant-independent joint realization of $q_{1}$ and $q_{2}$ to $b^{1}_{0} \ldots b^{1}_{n+1}$) gives us our desired $b^{2}_{n+1}$ (as the image of $b^{*}$ under this automorphism.)

Now having constructed the configuration, let $a_{n}$ realize the consistent set of instances of $r(x, y)$ coming from $J_{n}$, and let $d_{i}=(b^{1}_{i}, b_{i}^{2})$, $z=(y^{1}, y^{2})$, $\phi(x, y) = \rho(x, y_{1})$, $\psi(x, z) = \rho(x, y_{2})$. As in the proof of Theorem 7.17 of \cite{Co15}, these satisfy the hypotheses of the following fact:

\begin{fact}(Corrected version of proposition 7.2, \cite{Co15}\footnote{Gabriel Conant, in a personal communication with the author (\cite{Co23}), noted this correction to Proposition 7.2 of \cite{Co15}, and plans to publicize this in a future corrigendum. See also Observation 6.15 of \cite{Mal10b} for an earlier version of this fact, which can also be used here.})

Suppose there are sequences $\{a_{i}\}_{i < \omega}$, $\{d_{i}\}_{i < \omega}$, and $\phi (x, y)$, $\psi(x, y)$ so that

(i) $\models \varphi (a_{i}, d_{j})$ for all $i < j$ and $\psi (a_{i}, d_{j})$ for all $i \geq j$

(ii) for all $i < j$, $\varphi(x, b_{i}) \cup \psi(x, b_{j})$ is inconsistent

Then $T$ is $\mathrm{SOP}_{3}$.
\end{fact}

So $T$ is $\mathrm{SOP}_{3}$.

This concludes the proof of the main result of this paper.

\:

\textbf{Acknowledgements} The author would like to thank Mark Kamsma and Itay Kaplan, as well as seminar participants at Hebrew University of Jerusalem, Imperial College London, and Université Claude Bernard Lyon 1 for many helpful edits and comments.

\bibliographystyle{plain}
\bibliography{refs}

\begin{thebibliography}{10}

\bibitem{A09}
Hans Adler.
\newblock A geometric introduction to forking and thorn-forking.
\newblock {\em Journal of Mathematical Logic}, 9, 2009.

\bibitem{AK20}
JinHoo Ahn and Joonhee Kim.
\newblock $\mathrm{SOP}_1$, $\mathrm{SOP}_2$, and antichain tree property,
  preprint. Available at https://arxiv.org/abs/2003.10030. 2020.

\bibitem{AK21}
JinHoo Ahn, Joonhee Kim, and Junguk Lee.
\newblock On the antichain tree property, preprint. Available at
  https://arxiv.org/abs/2106.03779. 2021.

\bibitem{Che14}
Artem Chernikov.
\newblock Theories without the tree property of the second kind.
\newblock {\em Annals of Pure and Applied Logic}, 165(2):695--723, 2014.

\bibitem{C15}
Artem Chernikov.
\newblock $\mathrm{NTP}_{1}$ theories, presentation slides. Paris. Available at
  https://www.math.ucla.edu/~chernikov/slides/ParisSeminar2015.pdf. June 2015.

\bibitem{CK09}
Artem Chernikov and Itay Kaplan.
\newblock Forking and dividing in $\mathrm{NTP}_2$ theories.
\newblock {\em The Journal of Symbolic Logic}, 77(1):1--20, 2012.

\bibitem{CR15}
Artem Chernikov and Nicholas Ramsey.
\newblock On model-theoretic tree properties.
\newblock {\em Journal of Mathematical Logic}, 16(2):1650009, 2016.

\bibitem{Co15}
Gabriel Conant.
\newblock An axiomatic approach to free amalgamation.
\newblock {\em The Journal of Symbolic Logic}, 82(2):648–671, 2017.

\bibitem{Co23}
Gabriel Conant, Personal communication. 2023.

\bibitem{CoK19}
GABRIEL CONANT and ALEX KRUCKMAN.
\newblock Independence in generic incidence structures.
\newblock {\em The Journal of Symbolic Logic}, 84(2):750–780, 2019.

\bibitem{D19}
Christian D'Elbée.
\newblock Forking, imaginaries, and other features of $\text {ACFG}$.
\newblock {\em The Journal of Symbolic Logic}, 86(2):669–700, 2021.

\bibitem{DS04}
Mirna D\v{z}amonja and Saharon Shelah.
\newblock On $\lhd^{*}$-maximality.
\newblock {\em Annals of Pure and Applied Logic}, 125(1-3):119--158, 2004.

\bibitem{D18}
Christian d’Elb{\'e}e.
\newblock Generic expansions by a reduct.
\newblock {\em Journal of Mathematical Logic}, 21(03):2150016, 2021.

\bibitem{EW09}
David~E. Evans and Mark Wing~Ho Wong.
\newblock {Some remarks on generic structures}.
\newblock {\em Journal of Symbolic Logic}, 74(4):1143 -- 1154, 2009.

\bibitem{GIL02}
Rami~P. Grossberg, Jos{\'e} Iovino, and Olivier Lessmann.
\newblock A primer of simple theories.
\newblock {\em Archive for Mathematical Logic}, 41:541--580, 2002.

\bibitem{KR17}
Itay Kaplan and Nicholas Ramsey.
\newblock On kim-independence.
\newblock {\em Journal of the European Mathematical Society}, 22, 02 2017.

\bibitem{K09}
Byunghan Kim.
\newblock $\mathrm{NTP}_{1}$ theories, presentation slides. BIRS Workshop,
  Seoul. Available at https://www.birs.ca/workshops/2009/09w5113/files/Kim.pdf.
  February 2009.

\bibitem{KK11}
Byunghan Kim and Hyeung-Joon Kim.
\newblock Notions around tree property 1.
\newblock {\em Annals of Pure and Applied Logic}, 162(9):698--709, 2011.

\bibitem{KL22}
Joonhee Kim and Hyoyoon Lee.
\newblock Some remarks on kim-dividing in $\mathrm{NATP}$ theories, Preprint.
  Available at https://arxiv.org/pdf/2211.04213.pdf. 2022,.

\bibitem{Kr}
Alex Kruckman.
\newblock Research statement., Available at
  https://akruckman.faculty.wesleyan.edu/files/2019/07/researchstatement.pdf.

\bibitem{KM22}
Alex Kruckman and Scott Mutchnik, Personal communication. 2022.

\bibitem{KR18}
Alex Kruckman and Nicholas Ramsey.
\newblock Generic expansion and $\text{Skolemization}$ in $\mathrm{NSOP}_{1}$
  theories.
\newblock {\em Annals of Pure and Applied Logic}, 169(8):755--774, aug 2018.

\bibitem{Lee22}
Hyoyoon Lee, Personal communication. Feb. 10, 2023.

\bibitem{MS17}
Maryanthe Malliaris and Saharon Shelah.
\newblock Model-theoretic applications of cofinality spectrum problems.
\newblock {\em Israel Journal of Mathematics}, 220(2):947--1014, 2017.

\bibitem{Mal10b}
M.E. Malliaris.
\newblock Edge distribution and density in the characteristic sequence.
\newblock {\em Annals of Pure and Applied Logic}, 162(1):1--19, 2010.

\bibitem{GFA}
Scott Mutchnik.
\newblock Conant-independence in generalized free amalgamation theories,
  Preprint. Available at https://arxiv.org/abs/2210.07527. 2022.

\bibitem{INDNSOP3}
Scott Mutchnik.
\newblock Properties of independence in $\mathrm{NSOP}_{3}$ theories, Preprint.
  Available at https://arxiv.org/abs/2305.09908. 2023.

\bibitem{Sc15}
Lynn Scow.
\newblock Indiscernibles, em-types, and ramsey classes of trees.
\newblock {\em Notre Dame Journal of Formal Logic}, 56(3):429--447, 2015.

\bibitem{She95}
Saharon Shelah.
\newblock Toward classifying unstable theories.
\newblock {\em Annals of Pure and Applied Logic}, 80(3):229--255, 1996.

\bibitem{SD00}
Saharon Shelah and Alexander Usvyatsov.
\newblock More on $\mathrm{SOP}_{1}$ and $\mathrm{SOP}_{2}$.
\newblock {\em Annals of Pure and Applied Logic}, 155(1):16--31, 2008.

\bibitem{TT12}
Kota Takeuchi and Akito Tsuboi.
\newblock On the existence of indiscernible trees.
\newblock {\em Annals of Pure and Applied Logic}, 163(12):1891--1902, 2012.

\bibitem{ZT12}
Katrin Tent and Martin Ziegler.
\newblock On the isometry group of the $\text{Urysohn}$ space.
\newblock {\em Journal of the London Mathematical Society}, 87(1):289--303, nov
  2012.

\bibitem{BYC07}
Itaï~Ben Yaacov and Artem Chernikov.
\newblock An independence theorem for $\mathrm{NTP}_{2}$ theories.
\newblock {\em The Journal of Symbolic Logic}, 79(1):135–153, 2014.

\end{thebibliography}

\end{document}